\documentclass[12pt]{amsart}

\usepackage{amsfonts,amsmath,amsthm,amscd,amssymb,latexsym,amsbsy,caption}
\usepackage[noadjust]{cite}

\usepackage{tikz}

\newtheorem{theorem}{Theorem}[section]
\newtheorem{lemma}[theorem]{Lemma}
\newtheorem{corollary}[theorem]{Corollary}
\newtheorem{proposition}[theorem]{Proposition}

\newtheorem{conjecture}[theorem]{Conjecture}

\theoremstyle{definition}
\newtheorem{definition}[theorem]{Definition}

\theoremstyle{remark}

\numberwithin{equation}{section}
\allowdisplaybreaks[4]

\newcommand{\RR}{\mathbb{R}}

\makeatletter
\@namedef{subjclassname@2020}{%
  \textup{2020} Mathematics Subject Classification}
\makeatother

\newcommand{\acr}{\newline\indent}

\author{Oleksiy Dovgoshey}
\address{\textbf{Oleksiy Dovgoshey}\acr
Department of Theory of Functions\acr
Institute of Applied Mathematics and Mechanics\acr
of NASU, Slovyansk, Ukraine\newline\indent
and Department of Mathematics and Statistics, University of Turku, Turku, Finland.}
\email{oleksiy.dovgoshey@gmail.com, oleksiy.dovgoshey@utu.fi}

\title{Characterization of geodesic distance on infinite graphs}



\begin{document}

\begin{abstract}
Let $G$ be a connected graph and let $d_G$ be the geodesic distance on $V(G)$. The metric spaces $(V(G), d_{G})$ are characterized up to isometry for all finite connected $G$ by David C. Kay and Gary Chartrand in 1964. The main result of the paper exands this cha\-rac\-te\-ri\-za\-tion on the infinite connected graphs. We also prove that every metric space with integer distances between its points admits an isometric embedding into $(V(G), d_G)$ for suitable $G$.
\end{abstract}

\maketitle

\section{Introduction}

Let $u,v \in V(G)$ be vertices of a connected graph $G$. Then the geodesic distance $d_G (u,v)$ is the minimum size of paths joining $u$ and $v$ in $G$. There are some important interconnections between metric properties of the space $(V(G), d_G )$ and combinatorial properties of $G$, \cite{HN, Pel2013}. It should be noted here that almost all of these interconnections were found for finite graphs. In particular the following theorem was proved by David C. Kay and Gary Chartrand in \cite{Kay_Chartrand_1965}.

\begin{theorem}\label{th1.1}
Let $(M,d)$ be a finite metric space. Then $(M,d)$ is isometric to the space $(V(G), d_G)$ for some finite connected $G$ if and only if the distance between every two points of $M$ is an integer and, for arbitrary $x,z \in M$ with $d(x,z) \geqslant 2 $, there is $y \in M$ such that $y$ lies between $x$ and $z$.
\end{theorem}

The initial goal of the paper is to extend the above theorem to connected graphs of arbitrary infinite order.

The paper is organized as follows. Section~2 contains some definitions and facts from theory of metric spaces and graph theory.

The main results are given in Section~3. In Theorem~\ref{th3.6}, an extension of Theorem~\ref{th1.1} to arbitrary infinite graphs is proved.

In Theorem~\ref{th3.5} we prove that each metric space with integer distances between points admits an isometric embeddeding into $(V(G), d_G)$ for some connected $G$. A weak version of Theorem~\ref{th3.5} is given in Corollary~\ref{cor3.7} for metric spaces with arbitrary distances between points.

A graph-theoretical reformulation of Menger's results on isometric embeddings in the real line are given in Conjecture~\ref{con4.2} of Section~4.
An extremal property of pseudo-linear quadruples is presented as a property of induced cycles in Conjecture~\ref{con4.4}.

\section{Initial definitions and facts}

Let us recall some concepts from the theory of metric spaces and graph theory.

\begin{definition}\label{def2.0}
Let $X$ be a nonempty set. A \emph{metric} on $X$ is a function $d : X \times X \to \RR^+,$ $\RR^+ = [0, \infty)$, such that for all $x, y, z \in X$:
\begin{itemize}
\item [$(i)$] $d(x, y) = d(y, x)$, symmetry;

\item [$(ii)$] $(d(x, y) = 0) \Leftrightarrow (x = y)$, non-negativity;

\item [$(iii)$] $d(x, y) \leqslant d(x, z) + d(z, y)$, triangle inequality.
\end{itemize}
\end{definition}

The main isomorphisms of metric spaces are the so-called isometries of such spaces.

\begin{definition}\label{def2.1}
Metric spaces $(X, d)$ and $(Y, \rho)$ are \emph{isometric} if there is a bijective mapping $\Phi : X \to Y$ such that
\begin{equation}\label{def2.1_eq1}
d(x, y) = \rho(\Phi(x), \Phi(y))
\end{equation}
for all $x, y \in X.$ In this case we say that $\Phi$ is an \emph{isometry} of $(X, d)$ and $(Y, \rho)$.
\end{definition}

Let $(X, d)$ and $(Y, \rho)$ be metric spaces. Recall that $\Phi : X \to Y$ is an \emph{isometric embedding} of $(X, d)$ in $(Y, \rho)$ if \eqref{def2.1_eq1} holds for all $x, y \in X$. We say that $(X, d)$ is \emph{isometrically embedded} in $(Y, \rho)$ if there exists an isometric embedding $X \to Y$.

We will also use the concept of ``metric beetweennes''. In the theory of metric spaces, the notion of ``metric betweenness'', first appeared at Menger's paper \cite{Menger1928}.

\begin{definition}\label{def2.3}
Let $(X,d)$ be a metric space and let $x$, $y, z$ be points of $(X,d)$. One says that $y$ \emph{lies between} $x$ and $z$ if $x \neq y \neq z$ and
\begin{equation*}%\label{def2.4_eq1}
d (x, z) = d (x, y) + d (y, z).
\end{equation*}
\end{definition}

The ``betweenness'' relation is fundamental for the theory of geodesics on metric spaces (see, e.g., \cite{Papadopoulos2005}).
Characteristic properties of the ternary relations that are ``metric betweenness'' relations for (real-valued) metrics were determined by Wald in \cite{Wald1931}. Later, the problem of ``metrization'' of betweenness relations (not necessarily by real-valued metrics) was considered in \cite{MZ1995PAMS, M1977FM} and \cite{Simko1999}. An infinitesimal version of the metric betweenness was studied in \cite{BD2012} and \cite{DD2010}. Paper \cite{BDKP2017AASFM} contains an explicit construction of a minimal metric space $(Y, \rho)$ such that a metric space $(X,d)$ is isometrically embedded in $(Y, \rho)$ if, among any three distinct points of $X$, there is one that lies between the others. An elemental proof of some Menger's results was given in \cite{DD2009BRaIIoMS}.

A \emph{graph} is a pair $(V, E)$ consisting of a nonempty set $V$ and a (possibly empty) set $E$ whose elements are unordered pairs of different points from $V$. For a graph $G = (V, E)$, the sets $V = V (G)$ and $E = E(G)$ are called \emph{the set of vertices (or nodes)} and \emph{the set of
edges}, respectively. We say that $G$ is \emph{nonempty} if $E(G) \neq  \varnothing$. If $\{x, y\} \in E(G)$, then the vertices $x$ and $y$ are \emph{adjacent}. Recall that a \emph{path} is a nonempty graph $P$ whose vertices can be numbered so that
\begin{equation}\label{s2_eq0}
V(P) = \{ x_0, x_1, \ldots, x_k\}, \ E(P ) = \{ \{ x_0, x_1\}, \ldots, \{x_{k-1}, x_k\} \}
\end{equation}
where $k \geqslant 1.$
In this case we say that $P$ is a path joining $x_0$ and $x_k$ and write
$$
P= (x_0, x_1, \ldots, x_k).
$$

A graph $G$ is \emph{finite} if $V (G)$ is a finite set, $|V (G)| < \infty$. The cardinal number $|V(G)|$ is called the \emph{order} of $G$. The \emph{size} of $G$ is the cardinal number $|E(G)|$. In this paper we consider graphs with arbitrary finite or infinite order.

A finite graph is called the \emph{cycle} $C_n$ if there exists an enumeration $v_1, v_2, \ldots, v_n$ of its vertices such that
$$
( \{ v_i, v_j \} \in E(C_n)) \Leftrightarrow (|i - j| = 1 \quad \textrm{or} \quad |i - j| = n - 1).
$$
In this case we write $C_n = (v_1, \ldots, v_n, v_1)$.

A graph $H$ is a \emph{subgraph} of a graph $G$ if
$$
V (H) \subseteq V (G) \quad \textrm{and} \quad E(H) \subseteq E(G).
$$
We write $H \subseteq G$ if $H$ is a subgraph of $G$.

If $H_1$ and $H_2$ are subgraphs of $G$, then the union $H_1 \cup H_2$ is a subgraph of $G$ such that
$$
V(H_1 \cup H_2) = V(H_1) \cup V(H_2) \quad \textrm{and} \quad
E(H_1 \cup H_2) = E(H_1) \cup E(H_2).
$$

A graph $G$ is \emph{connected} if for every two distinct $u, v \in V (G)$ there is a path $P \subseteq G$ joining $u$ and $v$.

A \emph{walk} in a graph $G$ is a sequence of vertices $v_0, v_1, \ldots, v_n$ such that $\{v_{i-1}, v_i\} \in E(G)$ for every $i \in \{1, \ldots, n\}$.

\begin{lemma}\label{lem2.4}
Let $G_1$ and $G_2$ be connected subgraphs of a graph $G$. If
\begin{equation}\label{lem2.4_eq1}
V(G_1) \cap V(G_2) \neq  \varnothing
\end{equation}
holds, then the union $G_1 \cup G_2$ also is a connected subgraph of $G$.
\end{lemma}

\begin{proof}
The graph $G_1 \cup G_2$ evidently is connected if $|V(G_1 \cup G_2)| = 1$.
Let us consider two different points $x_1, x_2$ of $G_1 \cup G_2$. We must show that there is a path $P \subseteq G_1 \cup G_2$ joining $x_1$ and $x_2$. This is trivially valid if
\begin{equation}\label{lem2.4_preq1}
x_1, x_2 \in G_1 \quad \textrm{or} \quad x_1, x_2 \in G_2
\end{equation}
because $G_1$ and $G_2$ are connected.

Suppose that \eqref{lem2.4_preq1} does not hold. Then, without loss of generality, we may assume
\begin{equation}\label{lem2.4_preq2}
x_1 \in G_1 \quad \textrm{and} \quad x_2 \in G_2.
\end{equation}
Let \eqref{lem2.4_eq1} hold and let $y$ be a point of $V(G_1) \cap V(G_2)$.
It follows from~\eqref{lem2.4_preq2} and $y \in V(G_1) \cap V(G_2)$ that
\begin{equation}\label{lem2.4_preq3}
x_1, y \in V( G_1) \quad \textrm{and} \quad x_2, y \in V(G_2).
\end{equation}
Since $G_1$ and $G_2$ are connected, \eqref{lem2.4_preq3} implies the existence of paths $P_1 \subseteq G_1$ and $P_2 \subseteq G_2$ such that
$$
P_1 = (v_0, \ldots, v_k), \quad P_2 = (v_k, \ldots, v_{k+l}),
$$
with $v_0=x_1$, $v_k=y$, $v_{k+l} = x_2$, and $k \geqslant 1$, and $l \geqslant 1$. Then the sequence $v_0, \ldots, v_k, v_{k+1}, \ldots, v_{k+l}$ is a walk in $G_1 \cup G_2$ from $x_1$ to $x_2$. Since every walk from one vertex to another different vertex ``contains'' a path joining these vertices, there is a path joining $x_1$ and $x_2$ in the graph $G_1 \cup G_2$. Thus $G_1 \cup G_2$ is connected.
\end{proof}

Let us recall now the concept of geodesic distance on graphs.

\begin{definition}\label{def2.5}
Let $G$ be a connected graph. For each two vertices $u,v \in V(G)$ we define the \emph{geodesic distance} $d_G(u,v)$ as
\begin{equation}\label{def2.3_eq1}
d_G (u,v) = \left\{
\begin{array}{ll}
0, & \quad \hbox{if} \quad u=v, \\
|E(P)|, & \quad \hbox{otherwise}
\end{array}
\right.
\end{equation}
where $P$ is the path of the minimal order joining $u$ and $v$ in $G$.
\end{definition}

The following proposition is well-known but usually formulated without any proof. See, for example, \cite[p.~3]{Pel2013} or Remark~1.16 in \cite{Knauer2019}.

\begin{proposition}\label{pr2.4}
Let $G$ be a connected graph. Then the geodesic distance $d_G$ is a metric on $V(G)$.
\end{proposition}

\begin{proof}
It directly follows from~\eqref{def2.3_eq1} that the function $d_G$ is symmetric and, moreover, $d_G (u,v) =0$ holds iff $u=v$. Thus, by Definition~\ref{def2.0}, the function $d_G$ is a metric on $V(G)$ iff the triangle inequality
\begin{equation}\label{pr2.4_preq1}
d_G(u,v) \leqslant d_G (u,w) + d_G(w,v)
\end{equation}
holds for all $u,v,w \in V(G)$.

Inequality~\eqref{pr2.4_preq1} is trivially valid of if $u=v$, or $u =w$, or $w =v$. Suppose that $u,v$ and $w$ are pairwise distinct.

Let us consider two paths $P_{u,w} \subseteq G$ and $P_{w,v} \subseteq G$ such that: $P_{u,w}$ connects $u$ and $w$; $P_{w,v}$ connects $w$ and $v$;
\begin{equation}\label{pr2.4_preq2}
d_G(u,w) = | E(P_{u,w})|
\end{equation}
and
\begin{equation}\label{pr2.4_preq3}
d_G(w,v) = | E(P_{w,v})|.
\end{equation}
Write
\begin{equation}\label{pr2.4_preq4}
G_1 : = P_{u,w} \cup P_{w,v}.
\end{equation}
Then $w \in V( P_{u,w}) \cap V(P_{w,v})$ and, consequently, $G_1$ is a connected subgraph of $G$ by Lemma~\ref{lem2.4}. It follows from \eqref{pr2.4_preq4} that
$$
u,v \in V(G_1).
$$
Hence the geodesic distance $d_{G_1} (u,v)$ is correctly defined, and, by Definition~\ref{def2.5}, the inequality
\begin{equation}\label{pr2.4_preq5}
d_G(u,v) \leqslant d_{G_1} (u,w)
\end{equation}
holds.

Inequality~\eqref{pr2.4_preq5} implies that \eqref{pr2.4_preq1} holds if
\begin{equation}\label{pr2.4_preq6}
d_{G_1}(u,v) \leqslant d_G (u,w) + d_G(w,v).
\end{equation}

Let $P^{1}_{u,v}$ be a path joining $u$ and $v$ in $G_1$ such that
\begin{equation}\label{pr2.4_preq7}
d_G(u,v) =|E(P^1_{u,v})|.
\end{equation}

It follows directly from the definition of paths that
\begin{equation}\label{pr2.4_preq8}
|E(P_{u,w})| = |V (P_{u,w})| -1,
\end{equation}
\begin{equation}\label{pr2.4_preq9}
|E(P_{w,v})| = |V (P_{w,v})|-1,
\end{equation}
and
\begin{equation}\label{pr2.4_preq10}
|E(P^1_{u,v})| = |V (P^1_{u,v})|-1.
\end{equation}

Now using \eqref{pr2.4_preq2}--\eqref{pr2.4_preq3}, \eqref{pr2.4_preq4}, and \eqref{pr2.4_preq8}--\eqref{pr2.4_preq9} we obtain
\begin{equation}\label{pr2.4_preq11}
\begin{aligned}
d_G (u,w) + d_G (w,v) &= |V(P_{u,w})| + |V(P_{w,v})|-2 \\
& = |V(G_1)| + |V(P_{u,w}) \cap V(P_{w,v})|-2.
\end{aligned}
\end{equation}
Since $w$ belongs to $V(P_{u,w}) \cap V(P_{w,v})$, \eqref{pr2.4_preq11} implies
\begin{equation}\label{pr2.4_preq12}
d_G(u,w) + d_G(w,v) \geqslant |V (G_1)|-1.
\end{equation}
Similarly it follows from $P^1_{u,v} \subseteq G_1$ and \eqref{pr2.4_preq10} that
\begin{equation}\label{pr2.4_preq13}
d_{G_1} (u,v) = |V(P^1_{u,v})| -1 \leqslant |V(G_1)|-1.
\end{equation}
Inequality \eqref{pr2.4_preq1} follows from \eqref{pr2.4_preq5}, \eqref{pr2.4_preq12} and \eqref{pr2.4_preq13}.

The proof is completed.
\end{proof}

The following lemmas will be used in the next section of the paper.

\begin{lemma}\label{lem2.6}
Let $G$ be a connected graph and let $u,v \in V(G)$. Then the vertices $u$ and $v$ are adjacent if and only if the equality
\begin{equation}\label{lem2.6_eq1}
d_G (u,v) = 1
\end{equation}
holds.
\end{lemma}

\begin{proof}
It follows directly from the definition of paths that $u$ and $v$ are adjacent iff there is a path $P \subseteq G$ of the size $1$ joining $u$ and $v$ in $G$. Now using Definition~\ref{def2.5} we see that the existence of such path $P$ is equivalent to the validity of equality~\eqref{lem2.6_eq1}.
\end{proof}

\begin{lemma}\label{lem2.7}
Let $G$ be a connected graph, let $x$ and $z$ be different vertices of $G$, and let $P$ be a path of the minimal order joining $x$ and $z$. Then every $y \in V(P)$ lies between $x$ and $z$, whenever $x \neq y \neq z$.
\end{lemma}

\begin{proof}
Let $y \in V(P)$ and $x \neq y \neq z$ hold. We must prove that
\begin{equation}\label{lem2.7_eq1}
d_G (x,z) = d_G (x, y) + d_G(y,z).
\end{equation}
Let us consider the geodesic distance $d_P$ on the vertex set $V(P)$. Then the definition of this distance and the definition of the paths give us the equality
\begin{equation}\label{lem2.7_preq1}
d_P (x,z) = d_P (x,y) + d_P(y,z)
\end{equation}
for every $y \in V(P)$.

Definition~\ref{def2.5} implies the inequalities
\begin{equation}\label{lem2.7_preq2}
d_P (x,y) \geqslant d_G(x,y) \quad \textrm{and} \quad d_P(y,z) \geqslant d_G(y,z).
\end{equation}
Moreover, it follows from the definition of the path $P$ and Definition~\ref{def2.5} that
\begin{equation}\label{lem2.7_preq3}
d_P (x,z) = d_G(x,z).
\end{equation}
Using the triangle inequality, \eqref{lem2.7_preq1}, \eqref{lem2.7_preq2} and \eqref{lem2.7_preq3} we obtain
$$
d_G(x,z) \leqslant d_G(x,y) + d_G (y,z) \leqslant d_P(x,y) + d_P (y,z) = d_P(x,z) = d_G (x,z).
$$
Equality \eqref{lem2.7_eq1} follows.
\end{proof}

\section{Metric betweennes and geodesic distance on graphs}

The following theorem shows that the Kay--Chartrand characterization of the spaces $(V(G), d_G)$ is valid for connected graphs $G$ of arbitrary order.

In what follows we denote by $\mathbb{N}_0$ the set of all non-negative integers, $\mathbb{N}_0 = \{0, 1, 2, \ldots\}$.

\begin{theorem}\label{th3.6}
Let $(X, d)$ be a metric space.
Then the following statements are equivalent.

\begin{itemize}
\item [$(i)$] There exists a connected graph $G$ such that the metric spaces $(X, d)$ and $(V(G), d_G)$ are isometric.

\item [$(ii)$] The inclusion
\begin{equation}\label{th3.6_eq2}
\{d(p,q) : p,q \in X\} \subseteq\mathbb{N}_0
\end{equation}
holds and, for any two points $x,z \in X$ satisfying the inequality $d(x,z) \geqslant 2 $, there is $y \in X$ such that
\begin{equation}\label{th3.6_eq4}
x \neq y \neq z
\end{equation}
and
\begin{equation}\label{th3.6_eq5}
d(x,z) = d(x,y) + d(y,z).
\end{equation}
\end{itemize}
\end{theorem}

\begin{proof}
$(i) \Rightarrow (ii)$. Let \(G\) be a connected graph such that \((X, d)\) and \((V(G), d_G)\) are isometric. Then inclusion~\eqref{th3.6_eq2} follows from Definitions~\ref{def2.1} and \ref{def2.5}.

Let us consider arbitrary $x,z \in X$ such that $d(x,z) \geqslant 2$. We need to show that there exists a point \(y \in X\) such that \eqref{th3.6_eq4} and \eqref{th3.6_eq5} are satisfied. Let $\Phi\colon X \to V(G)$ be an isometry of the metric spaces $(X,d)$ and $(V(G), d_G)$. Then the inequality
\begin{equation}\label{th3.3_preq1}
d_G (\Phi(x), \Phi(z)) \geqslant 2
\end{equation}
holds by Definition~\ref{def2.1}. Let $P$ be a path in $G$ joining $\Phi(x)$ and $\Phi(z)$ such that
$$
d_G (\Phi(x), \Phi(z)) = |E(P)|.
$$ Inequality \eqref{th3.3_preq1} implies $|E(P)| \geqslant 2$. Consequently there is $w \in V(P)$ such that
\begin{equation}\label{th3.3_preq2}
\Phi(x) \neq w \neq \Phi(z).
\end{equation}
Now using Lemma~\ref{lem2.7} we obtain
\begin{equation}\label{th3.3_preq3}
d_G (\Phi(x), \Phi(z)) = d_G (\Phi(x), w) + d_G (w, \Phi(z)).
\end{equation}
Let $\Phi^{-1}: V(G) \to X$ be the inverse of the mapping $\Phi: X \to V(G)$. Then $\Phi$ and $\Phi^{-1}$ are isometries and, consequently, we have
\begin{equation}\label{th3.3_preq4}
d(x, z) = d_G (\Phi(x), \Phi(z))
\end{equation}
and
\begin{equation}\label{th3.3_preq5}
d_G (\Phi(x), w) = d(\Phi^{-1} \left( \Phi(x)\right), \Phi^{-1}(w)) = d(x, \Phi^{-1} (w)),
\end{equation}
and
\begin{equation}\label{th3.3_preq6}
d_G (w, \Phi(z)) = d(\Phi^{-1} (w), \Phi^{-1}(\Phi(z))) = d(\Phi^{-1} (w), z).
\end{equation}
Equalities \eqref{th3.3_preq3}--\eqref{th3.3_preq6} imply
\begin{equation}\label{th3.3_preq7}
d(x, z) = d(x, \Phi^{-1}(w)) + d(\Phi^{-1} (w), z),
\end{equation}
and, in addition we have
\begin{equation}\label{th3.3_preq8}
x \neq \Phi^{-1}(w) \neq z
\end{equation}
by \eqref{th3.3_preq2}. Now \eqref{th3.6_eq4} and \eqref{th3.6_eq5} follows from \eqref{th3.3_preq8} and \eqref{th3.3_preq7} with $y = \Phi^{-1} (w)$.

$(ii) \Rightarrow (i)$. Let statement $(ii)$ be valid. Then we consider a graph $G$ such that
\begin{equation}\label{th3.6_preq0}
V(G) = X
\end{equation}
and
\begin{equation}\label{th3.6_preq1}
\left(\{p, q\} \in E(G) \right) \Longleftrightarrow (d(p, q)=1)
\end{equation}
for all $p, q \in X$.

We claim that $G$ is a connected graph and that the equality
\begin{equation}\label{th3.6_preq2}
d=d_G
\end{equation}
is satisfied, which obviously implies \((i)\).

It follows directly from the definition of the connectivity that $G$ is connected if for any different $x,z \in V(G)$ there is a connected subgraph \(G_{x,z}\) of \(G\) such that $x \in V(G_{x,z})$ and $z \in V(G_{x,z})$. Let us consider arbitrary different $x,z \in V(G)$. Write
\begin{equation}\label{th3.6_preq3}
n:=d(x,z).
\end{equation}
It follows from~\eqref{th3.6_eq2} and \eqref{th3.6_preq0} that $n$ is a positive integer. We construct the required $G_{x,z} \subseteq G$ by induction on $n$.

If $n=1$, then $x$ and $z$ adjacent in $G$ by~\eqref{th3.6_preq1}. So we can define $G_{x,z}$ by
$$
V(G_{x,z}) := \{x,z\} \quad \textrm{and} \quad E(G_{x,z}):= \{\{x,z\}\}.
$$
Suppose that we can construct $G_{x, z}$ for all $n<n_1$, where $n_1 \geqslant 2$ and that
$$
d(x,z)=n_1
$$
holds. Then using $(ii)$ we can find $y \in V(G)$ satisfying \eqref{th3.6_eq4} and \eqref{th3.6_eq5}.

Now \eqref{th3.6_eq4} and \eqref{th3.6_eq5} give us
$$
1 \leqslant d(x,y) \leqslant n_1-1,
$$
and
$$
1 \leqslant d(y,z) \leqslant n_1-1.
$$
Consequently, by induction hypothesis, there are connected $G_{x,y} \subseteq G$ and $G_{y,z} \subseteq G $ such that $x,y \in V(G_{x,y})$ and $y, z \in V(G_{y, z})$. Since $y \in V(G_{x,y}) \cap V(G_{y,z})$, the union $G_{x,y} \cup G_{y,z}$ is a connected subgraph of $G$ by Lemma~\ref{lem2.4}.

Write
$$
G_{x,z}:= G_{x,y} \cup G_{y,z} \,,
$$
then $G_{x,z}$ is connected and $x,z \in G_{x,z}$.

Thus $G$ is a connected graph and, consequently the geodesic distance $d_G$ is a correctly defined metric on $V(G)$ by Proposition~\ref{pr2.4}.

To complete the proof we must show that \eqref{th3.6_preq2} holds. Let $x$ and $z$ be different points of $X$. Assume first that $d(x,z) =1$. Then $x$ and $y$ are adjacent in $G$. Hence, by Lemma~\ref{lem2.6}, we obtain
$$
d_G(x,z) =1.
$$
Thus, the equality
\begin{equation}\label{th3.6_preq6}
d(x,z) = d_G(x,z)
\end{equation}
holds if $d(x,z) =1$. Moreover, equality~\eqref{th3.6_preq6} holds by Definition~\ref{def2.0} when $x=z$.

Suppose now that
\begin{equation}\label{th3.6_preq7}
d(x,z) = n \geqslant 2.
\end{equation}
Then using statement $(ii)$ we can find some points $y_1, \ldots, y_{n+1} \in X$ such that
\begin{equation*}\label{th3.6_preq8}
y_1 = x_1, \quad y_{n+1} = z,
\end{equation*}
and
\begin{equation}\label{th3.6_preq9}
d(y_i, y_{i+1}) = 1
\end{equation}
for every $i \in \{ i, \ldots, n\}$, and
\begin{equation}\label{th3.6_preq10}
d(x,z) = \sum_{i=1}^n d(y_i, y_{i+1}).
\end{equation}
It was noted above that \eqref{th3.6_preq9} impllies
$$
d(y_i, y_{i+1}) =d_G (y_i, y_{i+1}), \quad i= 1, \ldots, n.
$$
Consequently we can rewrite \eqref{th3.6_preq10} in the form
$$
d(x,z) = d_G(x,y_1) + d_G(y_1, y_2) + \ldots + d_G(y_{n-1}, y_n) + d_G(y_n,z).
$$
The last equality and the triangle inequality give us the inequality
\begin{equation}\label{th3.6_preq11}
d(x,z) \geqslant d_G(x,z).
\end{equation}

Let $P$ be the path of the minimal order joining $x$ and $z$ in $G$,
\begin{equation}\label{th3.6_preq12}
V(P) = \{ x_0, x_1, \ldots, x_k\}, \ k \leqslant 1, \ \textrm{and} \ E(P ) = \{ \{ x_0, x_1\}, \ldots, \{x_{k-1}, x_k\} \},
\end{equation}
where $x_0=x$ and $x_n =z$. It follows from \eqref{th3.6_preq7} that
$$
d_G (x,z) \geqslant 2,
$$
because in opposite case $d_G (x,z) = d(x,z) =1$ contrary to \eqref{th3.6_preq7}. Thus $k \geqslant 2$ holds.

Using Definition~\ref{def2.5} we obtain the equality
$$
d_G(x,z) = d_G(x,x_1) + d_G(x_1, x_2) + \ldots + d_G (x_{k-2}, x_{k-1}) + d_G (x_{k-1}, z).
$$
Now we can rewrite it as
\begin{equation}\label{th3.6_preq13}
d_G(x,z) = d(x,x_1) + d(x_1, x_2) + \ldots + d(x_{k-2}, x_{k-1}) + d (x_{k-1}, z)
\end{equation}
because, for every $j \in \{ 0, \ldots, k-1 \}$, the vertices $x_j$ and $x_{j+1}$ are adjecent that implies $d_G (x_j, x_{j+1}) = d (x_j, x_{j+1})$. Equality \eqref{th3.6_preq13} and the triangle inequality imply
$$
d_G (x,z) \geqslant d(x,z).
$$
The last inequality and \eqref{th3.6_preq11} give us
$$
d(x,z) = d_G (x,z).
$$
Equality~\eqref{th3.6_preq2} follows.

The proof is completed.
\end{proof}

Analyzing the above proof we obtain the following clarification of Theorem~\ref{th3.6}.

\begin{theorem}\label{th3.11}
Let $(X,d)$ be a metric space and let
$$
\{d(p,q)\colon p,q \in X\} \subseteq \mathbb{N}_0.
$$
If, for any $x,z \in X$ satisfying $d(x,z) \geqslant 2$, there is $y \in X$ such that $y$ lies between $x$ and $z$, then the graph $G$ defined by $V(G) =X$ and
$$
\left( \{x,y\} \in E(G) \right) \Longleftrightarrow (d(x,y) = 1),
$$
is connected and the geodesic distance $d_G$ satisfies the equality
$d = d_G.$
\end{theorem}

The following theorem shows that metric spaces with integer distances between points are subspaces of the spaces $(V(G), d_G)$.

\begin{theorem}\label{th3.5}
Let $(X,d)$ be a metric space. Then the following statements are equivalent.
\begin{itemize}
\item [$(i)$] There is a connected graph $G$ such that $(X,d)$ is isometric to a subspace of $(V(G), d_G)$.

\item [$(ii)$] The inclusion
\begin{equation}\label{th3.5_eq1}
\{d(p,q)\colon p,q \in X\} \subseteq\mathbb{N}_0
\end{equation}
holds.
\end{itemize}
\end{theorem}

\begin{proof}
$(i) \Longrightarrow (ii)$. Let $(i)$ hold. Then $(ii)$ directly follows from Definition~\ref{def2.5}.

$(ii) \Longrightarrow (i)$. Suppose that inclusion \eqref{th3.5_eq1} holds. Let us define a set $X_2$ of two-point subsets of $X $ by the rule: $\{ x,y\} \in X_2$ iff
\begin{equation}\label{th3.5_preq1}
d(x,y) \geqslant 2
\end{equation}
and, for every $z \in X$,
\begin{equation}\label{th3.5_preq2}
d(x,y) < d(x,z) + d(z,y)
\end{equation}
whenever
\begin{equation}\label{th3.5_preq3}
x \neq y \neq z.
\end{equation}

If $X_2= \varnothing$, then $(i)$ follows from Theorem~\ref{th3.6}.

Let us consider the case when $ X_2 \neq  \varnothing$.
For each $\{ x,y\} \in X_2$ we define a path $P_{x,y} = (v_0, \ldots, v_k)$ such that:
\begin{itemize}
\item [$(i_1)$] $v_0=x$, $v_k =y$ and $v_i \notin X$ for any $i \in \{ 1, \ldots, k-1\}$;
\item [$(i_2)$] The equality $ k=d(x,y)$ holds;
\item [$(i_3)$] If $\{ x_1, y_1\} \in X_2$, $\{ x_2,y_2\} \in X_2$ and $\{x_1, y_1\} \neq \{x_2, y_2\}$, then the intersection of the sets
$$
\{u \in V(P_{x_1, y_1})\colon x_1 \neq u \neq y_1 \}
$$
and
$$
\{u \in V(P_{x_2, y_2})\colon x_2 \neq u \neq y_2 \}
$$
is empty.
\end{itemize}

Let us now define graph $G$ as follows:
\begin{equation}\label{th3.5_preq4}
V(G) = X \cup \left( \cup_{\{x,y\} \in X_2} V (P_{x,y}) \right)
\end{equation}
and, for all $u,v \in V(G)$, $ u$ and $v$ are adjacent iff either $u,v \in X$ and $d(u,v) =1 $, or $\{ u,v\} \in E (P_{x,y})$ for some $\{x,y\} \in X_2$.

We claim that $G$ is a connected graph.

\begin{figure}[htb]
\begin{tikzpicture}
\node at (-3, 3) {\( X= \{x_1, x_2, x_3\} \)};

\coordinate [label=below:$x_1$] (A) at (0,0);
\coordinate [label=above:$x_2$] (C) at (0,3);
\coordinate [label=below:$x_3$] (B) at (4,0);
\draw (A) -- node[below] {$4$} (B) -- node[above] {$5$} (C) -- node[left] {$3$} (A);

\draw [fill, black] (A) circle (1pt);
\draw [fill, black] (B) circle (1pt);
\draw [fill, black] (C) circle (1pt);

%%%%%%%%%%%%%%%%%%%%%%%%%%%%%%%%%%%%%%%%%
\node at (-3, -3) {\( G = C_{12} \)};

\coordinate [label=below:$x_1$] (A) at (0,-6);
\coordinate [label=above:$x_2$] (C) at (0,-3);
\coordinate [label=below:$x_3$] (B) at (4,-6);
\draw (A) -- node[below] { } (B) -- node[above] { } (C) -- node[left] { } (A);

\draw [fill, black] (A) circle (1pt);
\draw [fill, black] (B) circle (1pt);
\draw [fill, black] (C) circle (1pt);

\draw [fill, black] (0,-5) circle (1pt);
\draw [fill, black] (0,-4) circle (1pt);

\draw [fill, black] (1,-6) circle (1pt);
\draw [fill, black] (2,-6) circle (1pt);
\draw [fill, black] (3,-6) circle (1pt);

\draw [fill, black] (0,-5) circle (1pt);
\draw [fill, black] (0,-4) circle (1pt);
\draw [fill, black] (0,-3) circle (1pt);

\draw [fill, black] (0.8,-3.6) circle (1pt);
\draw [fill, black] (1.6,-4.2) circle (1pt);
\draw [fill, black] (2.4,-4.8) circle (1pt);
\draw [fill, black] (3.2,-5.4) circle (1pt);
\end{tikzpicture}
\caption{The Egyptian triangle \(\{x_1, x_2, x_3\}\) is isometrically embedded in the cycle $C_{12}$ endowed with the geodesic distance.}
\label{fig0}
\end{figure}
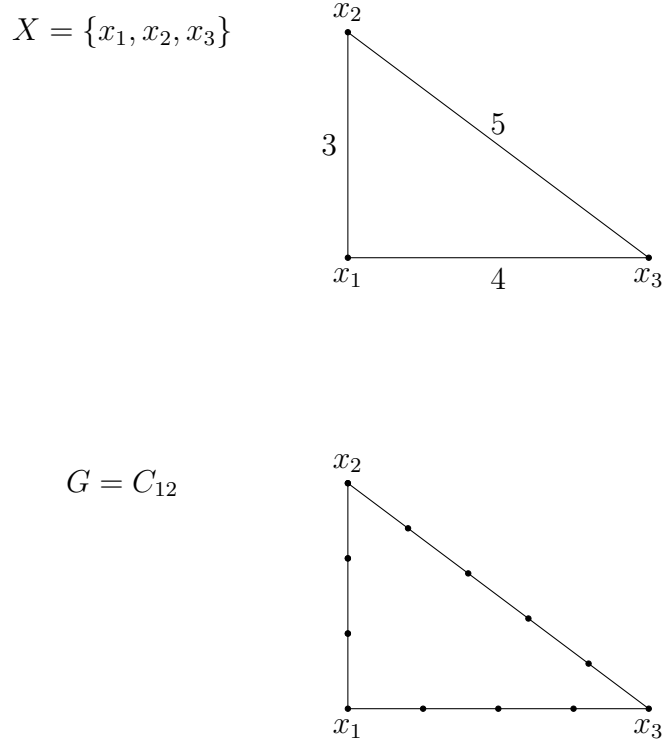

Let us consider two arbitrary distinct $u,v \in V(G)$. It is enough to show that there is a connected subgraph $G_{u,v}$ of $G$ such that $u \in V(G_{u,v})$ and $v \in V(G_{u,v})$ for the case when $u \in X$ and $v \in X$.

Indeed, if $ u \notin X$ and $v \notin X$, then, by \eqref{th3.5_preq4}, there are $\{ x_1, y_1\} \in X_2$ and $\{x_2, y_2\} \in X_2$ such that $u \in V (P_{x_1, y_1})$, and $v \in V( P_{x_2, y_2})$, and
$$
x_1, x_2, y_1, y_2 \in X.
$$
Without less of generality we may assume that $x_1 \neq y_2$. Suppose that there is a connected subgraph of $G_{x_1, y_2}$ such that
$$
x_1, y_2 \in V(G_{x_1, y_2}),
$$
then Lemma~\ref{lem2.4} implies that the union $P_{x_1, y_1} \cup P_{x_2, y_2} \cup G_{x_1, y_2}$ also is a connected subgraph of $G$. The cases $u \in X$, $v \notin X$ and $u \notin X$, $v \in X$ can be treated similarly, so we omit the details here.

Suppose now that $u,v \in X$ and $u \neq v$. If $d(u,v) =1 $, then, by definition, $u$ and $v$ are adjacent in $G$ and, consequently, we may take $G_{u,v}$ with
$$
V (G_{u,v}) = \{u,v\} \quad \textrm{and} \quad E(G_{u,v}) = \{ \{u,v\}\}.
$$
If
\begin{equation}\label{th3.5_preq5}
d(u,v) \geqslant 2
\end{equation}
holds and $\{u,v\} \in X_2$, then the path $P_{u,v}$ defined as in $(i_1) - (i_3)$ is a connected subgraph of $G$, and $u,v \in P_{u,v}$, so we can set $G_{u,v} := P_{u,v}$.

Let us consider now the case when \eqref{th3.5_preq5} is valid but $\{u,v\} \notin X_2$.
Then, by definition of $X_2$, there exist some points $p_0, p_1, \ldots, p_{n+1}$ such that
\begin{equation}\label{th3.5_preq6_0}
p_0=u \quad \textrm{and} \quad p_{n+1} =v,
\end{equation}
\begin{equation}\label{th3.5_preq6}
d(u,v) = \operatornamewithlimits{\sum}^{n}_{i=0} d(p_i, p_{i+1}),
\end{equation}
and, for each $i \in \{0, 1, \ldots, n+1\}$, we have
\begin{equation}\label{th3.5_preq7}
\textrm{either } \quad d(p_i, p_{i+1}) =1 \quad \textrm{or} \quad \{p_i, p_{i+1}\} \in X_2.
\end{equation}
The paths $P_{p_0,p_1}, P_{p_1, p_2}, \ldots, P_{p_n,p_{n+1}}$ are connected subgraphs of $G$. Consequently,
$$
G_{u,v} : = \operatornamewithlimits{\bigcup}_{i=0}^{n} P_{p_i, p_{i+1}}
$$
also is a connected subgraph of $G$ by Lemma~\ref{lem2.4}, and $u,v \in V(G_{u,v})$ holds by definition of $G_{u,v}$.

Thus $G$ is connected. To complete the proof we must show that the equality
\begin{equation}\label{th3.5_preq8}
d_G (u,v) = d(u,v)
\end{equation}
holds for all $u,v \in X$.

Equality~\eqref{th3.5_preq8} holds iff we have
\begin{equation}\label{th3.5_preq11}
d(u,v) \geqslant d_G (u,v)
\end{equation}
and
\begin{equation}\label{th3.5_preq12}
d(u,v) \leqslant d_G (u,v).
\end{equation}

Let us prove \eqref{th3.5_preq11}.

If $u=v$ holds, then inequality \eqref{th3.5_preq11} directly follows from Definition~\ref{def2.0}.

Let
\begin{equation}\label{th3.5_preq9}
d(u,v) = 1
\end{equation}
hold. Then, by definition of the graph $G$, we obtain
$$
\{u,v\} \in E(G).
$$
Hence, by Lemma~\ref{lem2.6} the equality
\begin{equation}\label{th3.5_preq10}
d_G (u,v) =1
\end{equation}
holds. Now \eqref{th3.5_preq11} follows from \eqref{th3.5_preq9} and \eqref{th3.5_preq10}.

Let us consider now the case when $u,v \in X$ and
$$
d(u,v) \geqslant 2
$$
holds.

Suppose that $\{u,v\} \in X_2$, and consider the path $P_{u,v}$ satisfying conditions $(i_1)-(i_3)$ with $x=u$ and $y=v$. Then we have
\begin{equation}\label{th3.5_preq13}
P_{u,v} \subseteq G
\end{equation}
by definition of $G$ and
\begin{equation}\label{th3.5_preq14}
d(u,v) = |E(P_{u,v})|
\end{equation}
by $(i_2)$. Definition~\ref{def2.5} and equality \eqref{th3.5_preq14} now imply \eqref{th3.5_preq11}.

If $\{u,v\} \notin X_2$, then there exist some points $p_0, \ldots, p_{n+1} \in X$ such that \eqref{th3.5_preq6_0}, \eqref{th3.5_preq6} and \eqref{th3.5_preq7} are valid. Using \eqref{th3.5_preq7} we obtain
\begin{equation}\label{th3.5_preq15}
d(u,p_1) \geqslant d_G (u,p), \ d(p_1, p_2) \geqslant d_G (p_1, p_2), \ \ldots, \ d(p_n, v) \geqslant d_G (p_n, v).
\end{equation}
Moreover, we have
\begin{equation}\label{th3.5_preq16}
d_G(u,v) \leqslant d_G (u,p_1) + d_G (p_1, p_2) + \ldots + d_G (p_n, v)
\end{equation}
by triangle inequality. Hence
$$
d(u,v) \geqslant d_G (u,p_1) + d_G (p_1, p_2) + \ldots + d_G (p_n, v) \geqslant d_G (u,v)
$$
holds by \eqref{th3.5_preq6}, \eqref{th3.5_preq15} and \eqref{th3.5_preq16}. Thus inequality \eqref{th3.5_preq11} is valid for all $u,v \in X$.

Let us prove \eqref{th3.5_preq12} for all $u,v \in X$. Reasoning as above we see that $d_G (u,v) = d(u,v)$ holds if $d_G(u,v) \leqslant 1$. Thus if there are $u,v \in X$ such that \eqref{th3.5_preq12} does not hold then there are $u^*, v^* \in X$ such that
\begin{equation}\label{th3.5_preq17}
d(u^*,v^*) > d_G (u^*,v^*)
\end{equation}
but we have
\begin{equation}\label{th3.5_preq18}
d(u,v) \leqslant d_G (u,v)
\end{equation}
whenever
\begin{equation}\label{th3.5_preq19}
d_G(u,v) < d_G (u^*,v^*).
\end{equation}
Let $u^*$ and $v^*$ satisfy the above conditions, and let $P^*_{u^*,v^*} = \{w_0, \ldots, w_m\}$ be a path in $G$ such that $w_0 = u^*, \ldots, w_m=v^*$ and
\begin{equation}\label{th3.5_preq20}
d_G(u^*,v^*) = |E(P^*_{u^*,v^*})|.
\end{equation}

The following two cases are possible:

We have
\begin{equation}\label{th3.5_preq21}
w_i \notin X
\end{equation}
whenever $i \in \{1, \ldots, m-1\}$;

There is $w_{i_0} \in V(P^*_{u^*, v^*})$ such that
\begin{equation}\label{th3.5_preq22}
w_{i_0} \in X
\end{equation}
and $i_0 \in \{1, \ldots, m-1\}$.

In the first case it follows from \eqref{th3.5_preq4} and \eqref{th3.5_preq21} that there is a path $P_{x,y}$ satisfying conditions $(i_1) -(i_3)$ such that
$$
w_1 \in V(P_{x,y}).
$$
Conditions $(i_1)$ and $(i_3)$ imply the equality
\begin{equation}\label{th3.5_preq24}
V(P_{x,y}) = V(P^*_{u^*,v^*}).
\end{equation}

Now using $(i_2)$ we see that \eqref{th3.5_preq22} implies
\begin{equation}\label{th3.5_preq25}
|E(P_{x,y})| = d(x,y) = d(u^*, v^*) = |E(P^*_{u^*, v^*})|.
\end{equation}
Equalities \eqref{th3.5_preq25} and \eqref{th3.5_preq20} give us
$$
d_G(u^*,v^*) =d (u^*, v^*)
$$
contrary to \eqref{th3.5_preq17}.

Let us consider the case when there is $w_{i_0} \in V(P^*_{u^*, v^*})$ such that \eqref{th3.5_preq22} holds and
\begin{equation}\label{th3.5_preq26}
i_0 \in \{ 1, \ldots, m-1\}.
\end{equation}
By Lemma~\ref{lem2.7} we have
\begin{equation}\label{th3.5_preq27}
d_G (u^*, v^*) = d_G (u^*, w_{i_0}) + d_G (w_{i_0}, v^*).
\end{equation}
Moreover \eqref{th3.5_preq26} implies that
\begin{equation}\label{th3.5_preq28}
u^* \neq w_{i_0} \neq v^*.
\end{equation}
Now it follows from \eqref{th3.5_preq27} and \eqref{th3.5_preq28} that
$$
d_G (u^*, w_{i_0}) < d_G (u^*, v^*)
$$
and
$$
d_G (w_{i_0}, v^*) < d_G (u^*, v^*).
$$
Hence, by definition of the points $u^*, v^*$, the equalities
\begin{equation}\label{th3.5_preq29}
d_G (u^*, w_{i_0}) = d (u^*, w_{i_0}),
\end{equation}
\begin{equation}\label{th3.5_preq30}
d_G (w_{i_0}, v^*) = d (w_{i_0}, v^*)
\end{equation}
hold. Now using \eqref{th3.5_preq27}, the triangle inequality in $(X,d)$, and equalities \eqref{th3.5_preq29}--\eqref{th3.5_preq30} we obtain
$$
d_G (u^*, v^*) = d (u^*, w_{i_0}) + d(w_{i_0}, v^*) \geqslant d (u^*, v^*),
$$
contrary to \eqref{th3.5_preq17}. Thus, \eqref{th3.5_preq12} holds for all $u,v \in X$.

The proof is completed.
\end{proof}

\begin{corollary}\label{cor3.6}
Let $(X,d)$ be a metric space. Then the following statements are equivalent.
\begin{itemize}
\item[$(i)$] There is a connected graph $G$ such that $(X,d)$ is isometrically embedded in $(V(G), d_G)$.

\item[$(ii)$] The distance between any two points of $X$ is an integer number.
\end{itemize}
\end{corollary}

\begin{corollary}\label{cor3.7}
For every metric space $(X,d)$ there exists a connected graph $G$ and an injective mapping $\Phi: X \to V(G)$ such that
\begin{equation}\label{cor3.7_eq1}
d(x,y) \leqslant d_G (\Phi(x), \Phi(y)) < d(x,y) +1
\end{equation}
for all $x,y \in X$.
\end{corollary}

\begin{proof}
Let $\mathbb{R}^+ \ni t \mapsto \lceil t \rceil \in \mathbb{N}_0$ be the sailing function,
\begin{equation}\label{cor3.7_preq1}
\lceil t \rceil = \min \{n \in \mathbb{N}_0 : n \geqslant t\}
\end{equation}
for each $t \in \mathbb{R}^+$. It is known that, for each metric space $(X,d)$, the function
\begin{equation}\label{cor3.7_preq2}
X \times X \ni (x,y) \mapsto \lceil d(x,y) \rceil \in \mathbb{R}^+
\end{equation}
remains a metric on $X$. (See, for example, \cite[Corollary~1, p.~6]{Dobos1998}). Write $ \lceil d \rceil $ for the metric on $X$ defined by \eqref{cor3.7_preq2}. Since $\lceil t \rceil$ is integer for every $t \in \mathbb{R}^+$, Corollary~\ref{cor3.6} implies that the metric space $(X, \lceil d \rceil )$ is isometrically embedded in $(V(G), d_G)$ for some connected graph $G$. Let $\Phi: X \to V(G)$ be an isometric embedding of $(X, \lceil d \rceil )$ in $(V(G), d_G)$. Then
\begin{equation}\label{cor3.7_preq3}
\lceil d (x,y) \rceil = \lceil d \rceil (x,y) = d_G (\Phi(x), \Phi(y))
\end{equation}
holds for all $x,y \in X$. Moreover, we have
\begin{equation}\label{cor3.7_preq4}
t \leqslant \lceil t \rceil < t+1
\end{equation}
for every $t \in \mathbb{R}^+$ by \eqref{cor3.7_preq1}.
Now \eqref{cor3.7_eq1} follows from \eqref{cor3.7_preq2}--\eqref{cor3.7_preq4}.
\end{proof}

\section{Conclusion. Expected results}

Let us denote by $\mathfrak{MB}$ the class of all metric spaces $(X,d)$ such that the equality
$$
d(x,z) = d(x,y) + d(y,z)
$$
holds whenever $d(x,z) \geqslant \max \{ d(x,y), d(y,z)\}$ .

The following is the particular case of the classical Menger's result on the isometric embeddings into Euclidean spaces.

\begin{theorem}\label{th4.1} {\rm\cite{Menger1928}}
Let $(Y, \rho) \in \mathfrak{M B}$ be a metric space with $|Y| \geqslant 5$. Then $(Y, \rho)$ is isometric to some subspace of $\mathbb{R}$.
\end{theorem}

It was also proved by K. Menger in \cite{Menger1928}, that a four-point metric space $(X,d) \in \mathfrak{M B}$ cannot be isometrically embedded in $\mathbb{R}$ if and only if the points of $X$ can be labelled $x_1, x_2, x_3, x_4$ such that
$$
d(x_1, x_2) = d(x_3, x_4) = s, \quad d(x_2, x_3) = d(x_1, x_4) = t,
$$
\begin{equation}\label{s4_eq1}
d(x_1, x_3) = d(x_2, x_4) = s + t,
\end{equation}
where $s$ and $t$ are some positive constants.

The ordered four-point metric spaces $\{x_1, x_2, x_3, x_4\}$ satisfying~\eqref{s4_eq1} are sometimes referred as \emph{pseudo-linear quadruples}. If \eqref{s4_eq1} holds with $s=t$, then we say that the corresponding $\{x_1, x_2, x_3, x_4\}$ is an \emph{equilateral} pseudo-linear quadruple.

Recall that an infinite graph $R$ of the form
\begin{align*}
V (R) &= \{v_1, v_2, \ldots , v_n, v_{n+1}, \ldots \},\\
E(R) &= \{ \{v_1, v_2\}, \ldots, \{ v_n, v_{n+1} \}, \ldots \}
\end{align*}
is called a \emph{ray}.

Moreover, a graph $DR$ is called a \emph{double ray} if
$$
V(DR) = \{ \ldots, v_{-2}, v_{-1}, v_0, v_1, v_2, \ldots \}
$$
and
$$
E(DR)= \{ \ldots, \{v_{-2}, v_{-1}\}, \{v_{-1}, v_0\}, \{v_0, v_1\}, \{v_1, v_2\}, \ldots \}.
$$

The following conjecture presents a reformulation of above-mentioned Menger's result in the language of graph theory.

\begin{conjecture}\label{con4.2}
Let $G$ be a nonempty connected graph.
Then $(V(G), d_G) \in \mathfrak{M B}$ if and only if one from the following statements hold.
\begin{itemize}
\item [$(i)$] $G$ is isomorphic to a path.

\item [$(ii)$] $G$ is isomorphic to cycle $C_{4}$.

\item [$(iii)$] $G$ is isomorphic to ray $R$.

\item [$(iv)$] $G$ is isomorphic to double ray $DR$.
\end{itemize}
\end{conjecture}

The following theorem is a special case of Theorem~1 of paper \cite{DP2011SMJ}.

\begin{theorem}\label{th4.3}
Let $(X, d)$ be a metric space and $\{x_1, x_2, x_3, x_4\}$ be an ordered four-point subspace of $(X,d)$. Write
$$
p =d(x_1, x_2) + d(x_2, x_3) + d(x_3, x_4) + d(x_4, x_1).
$$
Then we have
\begin{equation}\label{th4.3_eq1}
d(x_1, x_3) d(x_2, x_4) - d(x_1, x_2) d(x_3, x_4) - d(x_4, x_1) d(x_2, x_3) \leqslant \frac{p^2}{8}.
\end{equation}
Equality in \eqref{th4.3_eq1} is attained if and only if $\{ x_1, x_2, x_3, x_4 \}$ is an equilateral pseudo-linear quadruple.
\end{theorem}

Recall that a subgraph $H$ of a graph $G$ is called an \emph{induced subgraph} of $G$ if any two vertices $u,v $ of $H$ are adjacent in $H$, whenever $u,v$ are adjacent in $G$,
$$
\left( \{u,v\} \in E(H)\right) \Longleftrightarrow \left( \{u,v\} \in E(G) \right).
$$

The next conjecture is a reformulation of the second path of Theorem~\ref{th4.3} for case of geodesic distances on graphs.

\begin{conjecture}\label{con4.4}
Let $G$ be a nonempty connected graph and let $X$ be a four-point subspace of $(V(G), d_G)$. Then the following statements are equivalent.
\begin{enumerate}
\item [\((i)\)] If \(H\) is the induced subgraph of \(G\) and \(V(H) = X\), then \(H\) is a cycle.
\item [\((ii)\)] The points of \(X\) can be ordered such that the ordered four-point subspace \(X = \{x_1, x_2, x_3, x_4\}\) of $(V(G), d_G)$ is an equilateral pseudo-linear quadruple.
\end{enumerate}
\end{conjecture}

\section*{Acknowledgments}

The author was supported by grant 359772 of the Academy of Finland.

\section*{Financial disclosure}

None reported.

\section*{Conflict of interest}

The author declares that there is no potential conflict of interest.

\end{document}